\documentclass[article]{amsart}

\usepackage{amsmath}
\usepackage{amsfonts}
\usepackage{amsmath,mathtools}
\usepackage{amssymb}
\usepackage{url}
\usepackage{amscd}
\usepackage{bbm}
\usepackage{amsthm}
\usepackage{tikz}
\usepackage{graphicx}
\usetikzlibrary{matrix}
\usepackage{caption}
\usepackage{mathrsfs}
\numberwithin{equation}{section}

\theoremstyle{plain}
\newtheorem{theorem}{Theorem}
\newtheorem*{theorem*}{Theorem}

\newtheorem{proposition}[theorem]{Proposition}

\newtheorem*{conjecture*}{Conjecture}
\newtheorem*{lemma*}{Lemma}

\theoremstyle{definition}
\newtheorem{definition}{Definition}
\newtheorem*{definition*}{Definition}

\theoremstyle{remark}
\newtheorem{remark}{Remark}

\begin{document}
\date{}
\title{A Study of Weakly Discontinuous Solutions for Hyperbolic Differential Equations Based on Wavelet Transform Methods}
\title[A Study of Weakly Discontinuous Solutions]{A Study of Weakly Discontinuous Solutions for Hyperbolic Differential Equations Based on Wavelet Transform Methods}
\author{Shijie Gu}

\address{Department of Mathematics \& Statistics\\
University of Nevada, Reno\\
1664 N. Virginia Street
Reno, NV, USA}

\email{sgu@unr.edu}


\begin{abstract}
A new approach to prove the one-dimensional Cauchy problem’s weakly discontinuous solutions for hyperbolic PDEs are on the characteristics is discussed in this paper. To do so, I use wavelet singularity detection methods or WTMM [6] based on two-dimensional wavelet transform and combine it with the Lipschitz index to strengthen the detection.
\end{abstract}

\maketitle


\newcommand\sfrac[2]{{#1/#2}}

\newcommand\cont{\operatorname{cont}}
\newcommand\diff{\operatorname{diff}}

Wavelet singularity detection theories have been applied in many fields [5], such as signal and imaging process, especially wavelet denoise methods. Multiresolution analysis is very useful for identifying peaks and valleys of noisy signals. The key point of wavelet-based detection is wavelet transform. By using wavelet transform, the "peaks" or "cusps" are able to be smoothed, and appear to be local maxima and zeros (inflections) by using first derivative, second derivative or even higher derivatives wavelet transform convolution (this method is called wavelet transform modulus maxima or WTMM [6]). 

In this paper, we introduced a theoretical application of WTMM. During the process of application, we successfully found out that WTMM is a good tool for the study of discontinuous and weakly discontinuous problems in PDEs. In other words, one can use the property of WTMM to detect the positions of weakly discontinous solutions. This brand new approach will be introduced in section 3, and the orientation of WTMM is found by Proposition 3. Hopefully this research can shed light on the studies of discontinuous problems in PDEs. In section 2, the notations in two-dimensional wavelet transform are introduced; in section 4, we demonstrate a way to distinguish the "real" jump discontinuity and "fake" one. 

Note that weakly discontinuity is the specific property of hyperbolic PDEs; that means parabolic and elliptical PDEs do not have such property. So the topic of this paper concentrates on hyperbolic PDEs.

Our below definition of weakly discontinuity is inspired by Bhamra's definition [1, P. 156] of weak solution.
\begin{definition}
Let $\Gamma$ be a smooth curve on $xOt$ plane and $D$ be a smooth solution area containing $\Gamma$. Then, $\Gamma$ is the weakly discontinuous curve of equation
\begin{equation}
au_{xx}+2bu_{xt}+cu_{tt}+du_x+eu_t+gu=f
\end{equation}
where the coefficients are real numbers, if 
\begin{itemize}
\item[(1)] $u\in C^{1}(D)\cap C^{2}(D\backslash\Gamma)$;
\item[(2)] $u$ satisfies Eq.(1.1);
\item[(3)] the second partial derivatives of $u$ have discontinuity of first kind on the curve $\Gamma$.
\end{itemize}
\end{definition}

Applying wavelet singularity detection theories and WTMM, we can change the discontinuities into the local modulus maxima. Such technique follows a new way to prove the weak form of below well-known theorem in Cauchy problem,
\begin{theorem}[7, P. 133]
Let $\Gamma: x = x(t)$ be the weakly discontinuous curve of Eq.(1.1). Suppose the coefficients in
Eq.(1.1) satisfy 
\begin{equation}
b^2-ac >0, c>0, \forall (x,t)\in D,
\end{equation}
then $x = x(t)$ must satisfy
\begin{equation}
\frac{dx}{dt}=\lambda_1(x,t) \hspace{0.1in}\mbox{or}\hspace{0.1in} \frac{dx}{dt}=\lambda_2(x,t),
\end{equation}
where $\lambda_1, \lambda_2$ are the real roots to the equation
\begin{equation}
c \lambda^2 - 2b\lambda + a =0.
\end{equation}
\end{theorem}

To restrict the field of our study, we may assume the discontinuity of first kind is jump discontinuity (shock). Then $u(x,t)$ has a jump discontinuity (shock) along $\Gamma$, which means $u(x,t)$ is continuously differentiable in two parts $D_1$ and $D_2$ of the domain $D$, but with a jump discontinuity (shock) along the dividing smooth curve $\Gamma$. So the limits $u_{x}^{-}(t)=u_{x}(x^{-},t)=u_x(x(t)-0,t-0)$ and also $u_{x}^{+}(t)=u_{x}(x^{+},t)=u_x(x(t)+0,t+0)$ exist. Figure 1 illustrates this kind of condition.
\begin{figure}[h!]
         \centering
         \includegraphics[width=5.5cm,height=4.2cm]{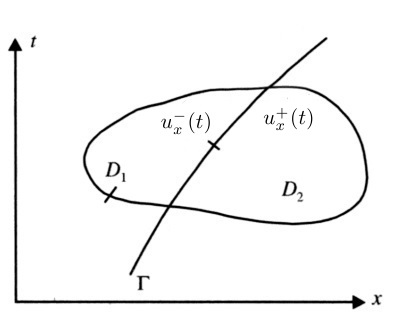}
         \caption{Smooth curve $\Gamma$ cross domain $D$ has jump discontinuities}
         \label{fig:1}
\end{figure}

Let $b=d=e=g=0$ in Eq.(1.1), then
\begin{equation}
au_{xx}+cu_{tt}=0.
\end{equation}

Consider Cauchy problem (or Cauchy initial condition and boundary condition), we have
\begin{equation}
\begin{cases}
u_{tt}-\nu^2u_{xx}=0,\\
u(x,0)=\phi(x),\\
u_t(x,0)=\psi(x).
\end{cases}
\end{equation}
where $\nu^2=-a/c$, $\phi(x) \in C^2, \psi(x) \in C^1$. Then the weak form of Theorem 1 is 
\begin{theorem}
For Eq. (1.6), the conclusion in Theorem 1 holds.
\end{theorem}

\section{Notations in Two-dimensional Wavelet Transform}
Before proving Theorem 2, we shall introduce the concept of two dimensional wavelet transform. First, let the unknown solution $u(x,t)\in L^2(\mathbb{R}^2)$ to Eq.(1.1) be a two dimensional signal, where $x$ and $t$ are displacement-coordinate and time-coordinate respectively. One can use the functions which have short support or rapid reduction ones to detect the locations of singular points of the solution $u(x,t)$. For convenience, we take Gaussian distribution function
\begin{equation}
\theta(x)=\frac{1}{\sqrt{2\pi}}\exp(\frac{-x^2}{2}).
\end{equation}
Since $u(x,t)$ can be understood as a two-dimensional signal, it would be necessary to employ the two-dimensional wavelet transform. That means one should extend one-dimensional Gaussian distribution function to two-dimensional Gaussian distribution function as 
\begin{equation}
\varPsi_{\sigma}(x,t)=\frac{1}{2\pi\sigma}\exp(-\frac{x^2+t^2}{2\sigma^2}),
\end{equation}
where $\sigma$ is a scale parameter.

Then denote $\varPsi_{\sigma;b_1, b_2}(x,t)$ as a function generated by $\varPsi_{\sigma}(x,t)$ with two-dimensional displacement,
\begin{equation}
\varPsi_{\sigma;b_1, b_2}(x,t)=\frac{1}{\sigma}\varPsi(\frac{x-b_1}{\sigma},\frac{t-b_2}{\sigma}).
\end{equation}
Apply the convolution integral to $u(x,t)$ and $\varPsi_{\sigma;b_1, b_2}(x,t)$,
\begin{equation}
\begin{split}
WTu(\sigma;b_1,b_2)&=\big\langle u(x,t),\varPsi_{\sigma;b_1, b_2}(x,t)\big\rangle\\
&=\frac{1}{\sigma}\int \!\!\! \int_D u(x,t)\varPsi(\frac{x-b_1}{\sigma},\frac{t-b_2}{\sigma})\,dx\,dt.
\end{split}
\end{equation}

Let $r:=(x,t)^{T}$ and $b:=(b_1,b_2)^{T}$, $r,b \in \mathbb{R}^2$. The smooth operator can be expressed by
\begin{equation}
P_{\sigma}\varPsi(r):=\frac{1}{\sigma^2}\varPsi(\frac{r}{\sigma}).
\end{equation}
Then, the convolution integral in Eq.(2.4) can be represented by
\begin{equation}
WTu(\sigma;b_1,b_2)=(P_{\sigma}\varPsi*u)(b).
\end{equation}
The muti-scale differential operator of the wavelet transform can be expressed by
\begin{equation}
WT^{(m)}u(\sigma;b_1,b_2)=(P_{\sigma}\varPsi^{(m)}*u)(b).
\end{equation}
\begin{remark}
In this paper, the first and second derivatives are sufficient to deal with the problem. Hence, we ignore the higher derivatives.
\end{remark}

Let 
\begin{equation}
\begin{split}
\varPsi_{x}^{(m)}=\frac{\partial^{m}}{\partial x^m}\varPsi_{\sigma;b_1,b_2}(x,t);\\
\varPsi_{t}^{(m)}=\frac{\partial^{m}}{\partial t^m}\varPsi_{\sigma;b_1,b_2}(x,t).
\end{split}
\end{equation}
where $m=1, 2$.

Denote 
\begin{equation}
\nabla\varPsi=\nabla\varPsi_{\sigma;b_1,b_2}(r)=
\left( \begin{array}{c}
\frac{\partial}{\partial x}\varPsi_{\sigma;b_1,b_2}(x,t)\\[0.1in]
\frac{\partial}{\partial t}\varPsi_{\sigma;b_1,b_2}(x,t)\end{array} \right),
\end{equation}
where $\nabla=(\frac{\partial}{\partial x},\frac{\partial}{\partial t})$.

\begin{remark}
Without the influence of the noise (wavelet transform is often influenced by noise at small scale and may produce false extreme points), the corresponding relationship between the local maxima or zeros of the wavelet transform modulus and singularity for $u(x,t)$ (signal) is very accurate. The smaller the scale parameter $\sigma$ is, the more accurate the corresponding relationship will be. Usually, we choose $\sigma = 2^j$ [6] The details will be discussed in section 4.
\end{remark}
\section{Proof of Theorem 2}
The proof of Theorem 2 can be divided into 2 cases---curve $\Gamma$ does intersect with $x$-axis and does not intersect with $x$-axis.
\begin{proof}[Proof of Case 1]
Without loss of generality, select a small neighborhood of discontinuous point $p_0=\big(x(t_0),t_0\big)\in \Gamma$.

If the curve $\Gamma$ intersects with $x$-axis at point $p_0=(x(0), 0)\in \Gamma$, then we can consider $u(x,0)=\phi(x)$ and get the modulus by using one-dimensional wavelet transform
\begin{equation}
|\theta'*\phi''|,
\end{equation}
where $|\cdot|$ is modulus.

Since $\phi''$ is discontinuous at $x_0$, 
then by WTMM, at $x=x(0)$, 
\begin{equation}
|\theta'*\phi''|\big|_{x=x(0)} 
\end{equation}
is maximum modulus in the small neighorhood of $p_0$. Figure 2 gives an example of 1D WTMM which may present a visualized way to help reader understand it well.
\begin{figure}[h!]
         \centering
         \includegraphics[width=7cm,height=5cm]{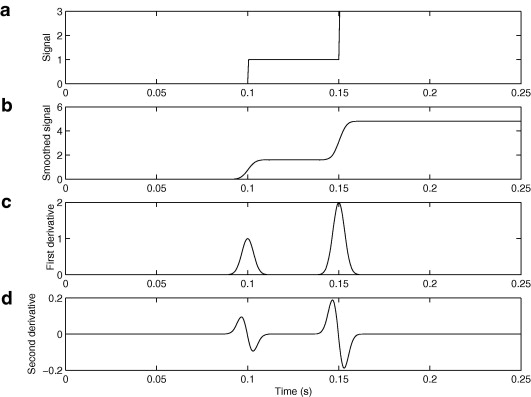}
         \caption{[2] Here's the process of WTMM for 1D signal. \textbf{a)} is the signal with jump discontinuities at $x=0.1$ and $x=0.15$;  \textbf{b)} discontinuities are smoothed when convoluted with $\theta$; \textbf{c)} discontinuities are local maxima when convoluted with $\theta'$; \textbf{d)} discontinuities are zeros when convoluted with $\theta''$.}
         \label{fig:2}
\end{figure}

Finally, apply d'Alembert formula, it is clear that the modulus maximum will spread along the characteristics: $x\pm \nu t = x(0)$. 
\end{proof}

\begin{proof}[Proof of Case 2]
If the curve $\Gamma$ has no intersection with $x$-axis, then we pick an arbitrary point $p_0=\big(x(t_0),t_0\big) \in \Gamma$. By hypothesis, the solution $u(p_0)=u\big(x(t_0),t_0\big)$ has second derivative jump discontinuity (shock). For the two-dimensional WTMM, inspired by Canny's [4] edge detector, Mallat and Hwang [6] showed that the components of the wavelet transform are proportional to the coordinates of the gradient vector of $u(x,t)$ smoothed by $\varPsi_{\sigma;b_1,b_2}(x,t)$; in other word,
\begin{equation}
\left( \begin{array}{c}
WT_{x}u(\sigma;b_1,b_2)\\[0.1in]
WT_{t}u(\sigma;b_1,b_2)\end{array} \right)
=\sigma\left( \begin{array}{c}
\frac{\partial}{\partial x}(P_{\sigma}\varPsi*u)\\[0.1in]
\frac{\partial}{\partial t}(P_{\sigma}\varPsi*u)\end{array} \right)
=\sigma \nabla (P_{\sigma}\varPsi*u).
\end{equation}
Canny [3] pointed out that the edge points where the modulus of the gradient vector of $P_{\sigma}\varPsi*u$ is maximum 
in the direction where the gradient vector points too. The
orientation of the gradient vector indicates the direction
where partial derivative of $u(x,t)$ has an absolute value
which is maximum. Mallat and Hwang [6] showed that Eq.(2.14) implies
an angle between the horizontal and gradient vector, that is
\begin{equation}
Au(\sigma;b_1,b_2 )=\arctan \left(\frac{WT_{t}u(\sigma;b_1,b_2)}{WT_{x}u(\sigma;b_1,b_2)}\right)
\end{equation}
Furthermore, the locally modulus maximum of $u(x,t)$ is along the gradient direction with the angle $Au(\sigma;b_1,b_2 )$ if $\sigma = 2^j$.  

In our proof, for $u_x$, let $(u_x)_t$ represent the partial derivative of $u_x$ along $\Gamma$ respect to $t$. Let $M$ be the modulus maximum operator of two-dimensional WTMM. Consider a small neighborhood of point $p_0=(x(t_0),t_0)\in \Gamma$, we have 
 \begin{equation}
M(u_x)_{t}=\sqrt{|WT_{x}(u_x)_{t}|^2+|WT_{t}(u_x)_{t}|^2},
\end{equation}
is modulus maximum in the small neigborhood of $p_0$. 

Then, when $t=t_0$, the partial derivative of $M(u_x)_{t}$ respect to $t$ is
\begin{equation}
\frac{\partial}{\partial t}(M(u_x)_{t})|_{t=t_0}=0.
\end{equation}
Similarly, at $t=t_0$, we have
\begin{equation}
M(u_t)_{t}=\sqrt{|WT_{x}(u_t)_{t}|^2+|WT_{t}(u_t)_{t}|^2},
\end{equation}
is modulus maxima in the small neigborhood of $p_0$. And
\begin{equation}
\frac{\partial}{\partial t}(M(u_t)_{t})|_{t=t_0}=0.
\end{equation}

Note that direction of WTMM is same with the gradiant vector, and $x$ is a function of $t$. Without loss of generality, we may assume $u_{x}^{-}(t_0)<u_{x}^{+}(t_0)$. Since there exists WTMM at $t=t_0$, we have
\begin{equation}
\frac{\partial}{\partial t}(WT_t(u_x)_{t})|_{t=t_0}=WT_{tt}(u_x)_{t}|_{t=t_0}=0.
\end{equation}
Consider the derivatives respect to $t$ along the curve $\Gamma$. (This direction is the WTMM orientation, since $x$ is a function of $t$). Then we have $(u_x)_t=u_{xx}x'(t)+u_{xt}$ and $(u_t)_t=(u_x)_{t}x'(t)+u_{tt}$. It implies the following systems,
\begin{equation}
\begin{cases}
x'(t_0)(WT_{tt}u_{xx})|_{t=t_0}+(WT_{tt}u_{xt})|_{t=t_0} =0,\\
x'(t_0)(WT_{tt}u_{xt})|_{t=t_0}+(WT_{tt}u_{tt})|_{t=t_0} =0.
\end{cases}
\end{equation}
In words,  as $t\rightarrow t_0$, the wavelet transform values of $(u_x)_t$ and $(u_t)_t$ on $D_1$ and $D_2$ are approaching to the values at $t=t_0$ on $\Gamma$. 

Next, use the convolution integral and wavelet transform on each side of Eq.(1.5), we have
\begin{equation}
a(WT_{tt}u_{xx})|_{t=t_0}+c(WT_{tt}u_{tt})|_{t=t_0}=0.
\end{equation}
By Eq.(3.10)-(3.11), we have
\begin{equation}
\begin{cases}
a(WT_{tt}u_{xx})|_{t=t_0}+c(WT_{tt}u_{tt})|_{t=t_0}=0,\\
x'(t_0)(WT_{tt}u_{xx})|_{t=t_0}+(WT_{tt}u_{xt})|_{t=t_0} =0,\\
x'(t_0)(WT_{tt}u_{xt})|_{t=t_0}+(WT_{tt}u_{tt})|_{t=t_0} =0.
\end{cases}
\end{equation}
Eq.(3.12) has nonzero solution, for the direction of derivatives of $WT_{t}u_{xx}$ and $WT_{t}u_{tt}$ respect to $t$ may not be always along $\Gamma$. Thus, determinant of Eq.(3.12) is equal to zero. By the arbitrary of $t_0$, we have $x'(t)=\nu $ or $x'(t)=-\nu$.
\end{proof}

Furthermore, one can also show that for the equation
\begin{equation}
au_{xx}+2bu_{xt}+cu_{tt}=0, \hspace{0.1in} (b \neq 0),
\end{equation}
the weakly discontinuity will also spread along the characteristics.

According to the above discussion, we have a conclusion for one-dimensional homogeneous hyperbolic PDE with Cauchy problem under the conditions mentioned in section 1, that is 
\begin{proposition}
Suppose there exists a weakly discontinuous curve $\Gamma: x=x(t)$ on the solutions, then the orientation of two-dimensional wavelet transform modulus maxima are along the characteristics.
\end{proposition}
\begin{proof}[Sketch of proof]
From the proof of Theorem 2, we can see that wavelet transform modulus maxima appear on the weakly discontinuities. Apply Theorem 1, discontinuous solutions are spreading along the characteristics. That means wavelet transform modulus maxima are along the characteristics. 
\end{proof}

\section{Lipschitz Index Criterion}
We cannot assure at all the points where signal changes rapidly are the discontinuities. So it is necessary to distinguish them by using Lipshitz index. For example, if $\Gamma$ intersects with $x$-axis at point $x_0$, the discontinuous points on $\Gamma$ are jump discontinuity (shock),
\begin{equation}
\lim_{x\rightarrow x_{0}^{-}}\phi''(x) \neq \lim_{x\rightarrow x_{0}^{+}}\phi''(x).
\end{equation}
Approximate $\phi''(x)$ at point $x_0$ by a step function, then in a small neighborhood of discontinuous point $x_0$, we have
\begin{equation}
|\phi''(x)-\phi''(x_0)|=\mathcal{O}(|x-x_0|^0).
\end{equation}
Lipschitz index $\alpha=0$.

Then, for $|\theta'*\phi''(x)-\theta'*\phi''(x_0)|$, we have
\begin{equation}
\begin{split}
&|\theta'*\phi''(x)-\theta'*\phi''(x_0)|\\
&=\frac{1}{\sqrt{2\pi}}\bigg|\int_{0}^{x}\exp(\frac{-\tau^2}{2})\phi''(x-\tau)d\tau
-\int_{0}^{x_0}\exp(\frac{-\tau^2}{2})\phi''(x_0-\tau)d\tau\bigg|\\
&\leq \frac{1}{\sqrt{2\pi}}\sup_{\tau>0} |\phi''(x-\tau)-\phi''(x_0-\tau)|\bigg|\int_{0}^{\infty}\exp(\frac{-\tau^2}{2})d\tau \bigg|\\
&=\frac{1}{2}\sup_{\tau>0}|\phi''(x-\tau)-\phi''(x_0-\tau)|.
\end{split}
\end{equation}
Since the derivative of $\phi''(x_0)$ is a Dirac function, by (3.3), we have
\begin{equation}
|\theta'*\phi''(x)-\theta'*\phi''(x_0)|\leq \mathcal{O}(|x-x_0|^{-1}).
\end{equation}
Lipschitz index $\alpha=-1$.

Hence, the solution of $u(x,t)$ at the discontinuous point $x_0$ will appear “shock” which is quite different from other point in the neighborhood. And the solutions on the characteristics also have the same feature.

Moreover, coefficients of wavelet transform in the interval section $[x_1, x_2]$ satisfy [6]
\begin{equation}
|W_f(\sigma, x)|\leq K \sigma ^\alpha \hspace{0.1in} \mbox{or} \hspace{0.1in} \log|W_f(\sigma, x)|\leq \log K + \alpha 
\log \sigma.
\end{equation}
When $\alpha <0$, maxima of wavelet transform modulus decreases with increasing $\sigma$. Hence, small scale parameter can avoid the confusion of the images of singularity and point which has rapid change. 

If $\Gamma$ does not intersect with $x$-axis, one should consider two-dimensional signal. The details about the two-dimensional WTMM Lipschitz index can be found in [6].

\section{Conclusion}
An approach to prove Theorem 2 by using WTMM is given in this paper. The main idea is to use convolution to smooth the solution containing weakly discontinuous points, and change the problem into finding the modulus maxima. The orientation of WTMM is also given in Proposition 3.

\end{document}